\documentclass[10pt,twocolumn]{article}

\usepackage[utf8]{inputenc}
\usepackage{amsmath}
\usepackage{amssymb}
\usepackage{amsthm}
\usepackage{titlesec}
\usepackage{cite}
\usepackage{color}
\usepackage{booktabs}
\usepackage{graphicx}
\usepackage{nicefrac}
\usepackage[colorlinks=false]{hyperref}
\usepackage[format=plain,labelfont=it]{caption}
\usepackage[left=1.5cm,right=1.5cm,top=2cm,bottom=2cm]{geometry}
\usepackage{algpseudocode}
\usepackage{csquotes}
\usepackage[font=footnotesize]{caption}
\usepackage{subcaption}

\usepackage{tikz}
\usetikzlibrary{arrows,positioning,patterns,decorations.pathreplacing}
\usepackage{tikz-3dplot}
\usepackage{pgfplots}
\usepgfplotslibrary{patchplots}
\usetikzlibrary{arrows}
\usetikzlibrary{decorations.pathmorphing}
\usetikzlibrary{decorations.shapes}
\usetikzlibrary{decorations.text}
\usetikzlibrary{positioning, shapes}

\renewcommand{\thesubsubsection}{\thesubsubsection.\arabic{subsubsection}}

\titleformat{\section}{\centering\normalfont\scshape}{\Roman{section}.}{5pt}{}
\titleformat{\subsection}{\normalfont\it}{\Alph{subsection}.}{5pt}{}
\titleformat{\subsubsection}[runin]{\normalfont\it}{\arabic{subsubsection})}{5pt}{}

\newcommand\infoFootnote[1]{%
  \begingroup
  \renewcommand\thefootnote{}\footnote{#1}%
  \addtocounter{footnote}{-1}%
  \endgroup}

\newtheorem{thm}{Theorem}
\newtheorem{cor}[thm]{Corollary}
\newtheorem{lem}[thm]{Lemma}

\newtheorem{assum}{Assumption}

\newtheorem{rem}{Remark}

\newcommand{\N}{\mathbb{N}} 

\newcommand{\R}{\mathbb{R}}

\newcommand{\Fc}{\mathcal{F}}

\newcommand{\Nc}{\mathcal{N}}

\newcommand{\Pc}{\mathcal{P}} 
\newcommand{\Qc}{\mathcal{Q}} 
\newcommand{\Rc}{\mathcal{R}} 
\newcommand{\Sc}{\mathcal{S}} 

\newcommand{\Uc}{\mathcal{U}}

\newcommand{\Xc}{\mathcal{X}}

\newcommand{\Zc}{\mathcal{Z}}

\newcommand{\Eb}{\boldsymbol{E}}
\newcommand{\Fb}{\boldsymbol{F}}
\newcommand{\Gb}{\boldsymbol{G}}
\newcommand{\Hb}{\boldsymbol{H}}
\newcommand{\Ib}{\boldsymbol{I}}

\newcommand{\Mb}{\boldsymbol{M}}
\newcommand{\Nb}{\boldsymbol{N}}

\newcommand{\Qb}{\boldsymbol{Q}}
\newcommand{\Rb}{\boldsymbol{R}}

\newcommand{\Wb}{\boldsymbol{W}}

\newcommand{\bb}{\boldsymbol{b}}

\newcommand{\db}{\boldsymbol{d}}

\newcommand{\fb}{\boldsymbol{f}}
\newcommand{\gb}{\boldsymbol{g}}

\renewcommand{\sb}{\boldsymbol{s}}

\newcommand{\ub}{\boldsymbol{u}}
\newcommand{\vb}{\boldsymbol{v}}

\newcommand{\xb}{\boldsymbol{x}}
\newcommand{\yb}{\boldsymbol{y}}
\newcommand{\zb}{\boldsymbol{z}}

\newcommand{\oneb}{\boldsymbol{1}}
\newcommand{\zerob}{\boldsymbol{0}}

\newcommand{\xib}{\boldsymbol{\xi}}

\newcommand{\varphib}{\boldsymbol{\varphi}}

\newcommand{\mub}{\boldsymbol{\mu}}

\newcommand{\Phib}{\boldsymbol{\Phi}}

\newcommand{\thetab}{\boldsymbol{\theta}}

\newcommand{\norm}[1]{\left\lVert#1\right\rVert}

\newcommand{\interior}{\mathrm{int}}

\DeclareMathAlphabet\mathbfcal{OMS}{cmsy}{b}{n}

\title{\vspace{-2mm}\bf Towards learning-enhanced solvers for multi-parametric optimization$^\ast$}
\author{Dieter Teichrib, Fabian K\"onig, and Moritz Schulze Darup\vspace{2mm}}
\date{}

\begin{document}

\maketitle

\textbf{\textit{Abstract}.} {\bf Neural networks (NNs) have been employed in various ways to accelerate the solution of multi-parametric optimization problems (OPs), which prominently arise in model predictive control (MPC). Early approaches focused on approximating the parametric solution (e.g., the optimal control law) directly by NNs. More recently, learning-based techniques have been used to emulate or enhance solver iterations. For the direct approximation of the parametric solution, it has been demonstrated that exploiting structural properties of the underlying OP facilitates the design of suitable NN architectures. In this paper, we show that structural analysis can also be leveraged to systematically integrate learning into classical optimization algorithms, thereby enhancing their performance. 

We formalize our concept for the case of multi-parametric quadratic programs (mpQPs), which frequently arise in linear-quadratic MPC, as a representative application.
More precisely, we demonstrate that tailored NNs enable a parameter-dependent fine-tuning of the step sizes in a standard dual projected gradient scheme. The design of the NNs is guided by the observation that optimal step sizes, which permit obtaining the exact optimal solution in finitely many iterations, are piecewise constant functions over the parameter space (e.g., the state space in MPC). We establish key theoretical properties of the proposed approach and illustrate its effectiveness in a numerical case study.}
\infoFootnote{All authors are with the \href{https://rcs.mb.tu-dortmund.de/}{Control and~Cyber-physical Systems Group}, Faculty of Mechanical Engineering, TU Dortmund University, Germany. E-mails:  \href{mailto:dieter.teichrib@tu-dortmund.de}{\{dieter.teichrib}, fabian2.koenig, moritz.schulzedarup\}@tu-dortmund.de. 

This work was supported by the Deutsche Forschungsgemeinschaft (DFG, German Research Foundation) under Grant 556837821.\vspace{0.5mm}}
\infoFootnote{\hspace{-1.5mm}$^\ast$This paper is a \textbf{preprint} of a contribution to the 65th IEEE Conference on Decision and Control 2026.}

\section{Introduction}

The computational complexity of multi-parametric optimization problems (OPs) remains a key challenge in applications that require real-time solutions. As is the case in, e.g., model predictive control (MPC) \cite{Rawlings2017}, where a multi-parametric OP, with the system state as a parameter, has to be solved at every time step within the system's sampling period to obtain the control input. To mitigate the challenges associated with the online solution of multi-parametric OPs, several approaches have been proposed. A non-machine-learning-based method is to solve the multi-parametric OP offline for all feasible parameter values and store the associated solutions. For cases where the multi-parametric OP is a multi-parametric quadratic program (mpQP), it is known that the resulting explicit optimizer is a piecewise affine (PWA) function of the parameter defined over polyhedral sets \cite{Bemporad2002}. But even evaluating the explicit PWA optimizer may be intractable in some cases, since the number of polyhedral sets grows in the worst case exponentially with the number of constraints in the OP \cite{Bemporad2002}.     

Therefore, neural networks (NNs) have been applied in various ways to accelerate the solution of multi-parametric OPs. The earliest applications of NNs in this context were primarily focused on directly approximating the explicit solution, i.e., training NNs to represent explicit optimizers \cite{Chen2018,Karg2020,Fabiani2022,Teichrib2024CDC} or optimal value functions (OVFs) \cite{Bhardwaj2020,Zhong2013,Chen2021}. More recently, learning-based techniques that aim to emulate or enhance solver iterations \cite{Ichnowski2021,Luken2024}, often referred to as "Learn-to-Optimize" (L2O) frameworks \cite{Chen2022,kotary2024,Luken2025}, were introduced. Now, an insight from the research on the direct approximation of solutions for multi-parametric OPs is that the efficacy of these approximations can be significantly enhanced when the NN topology is designed to exploit structural properties of the underlying OP. More specifically, the PWA and PWQ structure of the optimal control law and the OVF, respectively, in MPC for linear systems with quadratic costs have been used to derive NN topologies that yield low-error approximations \cite{Karg2020,Teichrib2023,Teichrib2025OJCSYS}. 

In this paper, we demonstrate that a rigorous analysis of classical optimization algorithms can also be used to systematically integrate structural knowledge into NNs that are used to approximate solver iterations. We formalize this concept for the case of mpQPs. Specifically, we show that tailored NNs can be used to enable parameter-dependent fine-tuning of step sizes within a standard dual projected gradient scheme. Our approach builds on the following results that we derived from a theoretical analysis of dual projected gradient solvers for mpQPs. First, we show that there exist optimal step sizes that, in principle, permit obtaining the exact optimal solution in a finite number of iterations (Theorem~\ref{thm:exact_solution} and Theorem~\ref{thm:centralResult}). Moreover, we prove that these optimal step sizes are piecewise constant (PWC) functions over the parameter space (e.g., the state space in the context of MPC). By designing NNs that structurally match this PWC mapping, we enhance the performance of the iterative dual projected gradient solver compared to a solver with a parameter-independent fixed step size. We first state the key theoretical results for this approach and demonstrate its effectiveness through a numerical case study.

The paper is organized as follows. In the remainder of this section, we state some basic notation. In Section~\ref{sec:Preliminaries}, we briefly summarize relevant background on mpQPs, dual projected gradient solver, optimal solutions derived from the Karush-Kuhn-Tucker (KKT) conditions, and NNs. In Section~\ref{sec:opt_steps}, we introduce some preliminary results that allow us to derive optimal step sizes for a dual projected gradient solver that exactly solves mpQPs under certain technical assumptions. Subsequently, we show in Section~\ref{sec:validityAssumptions} that it is possible to satisfy these assumptions for almost all feasible parameters of the mpQP.
Finally, in Section~\ref{sec:NNs_for_DPGD}, we combine all theoretical results to derive a learning-enhanced dual projected gradient solver. Moreover, we illustrate our method and evaluate its performance in Section \ref{sec:NNs_for_DPGD:A}.

\subsection{Notation}

The matrix $\Ib\in\R^{n \times n}$ is an identity matrix, where the dimension $n$ is clear from the context. Throughout the paper, we use the notation $\fb^i$ and $\mub^i$ for functions and matrices, respectively, to differentiate between different instances of the variables. For a polyhedral set $\Pc:=\{\xb\in\R^n \ | \ \Hb^\Pc\xb\leq\bb^\Pc\}$ we define the interior of $\Pc$ as $\interior(\Pc):=\{\xb\in\R^n \ | \ \Hb^\Pc\xb<\bb^\Pc\}$.

\section{Preliminaries and background}\label{sec:Preliminaries}

\subsection{Multi-parametric quadratic programs and their duals}

We investigate multi-parametric quadratic programs (QPs) of the form
\begin{equation}
\label{eq:mpQP_primal}
    \min_{\zb}\, \frac{1}{2}\zb^\top\Hb\zb + \xb^{\top}\Fb^{\top} \zb \quad \text{s.t.} \quad \Gb \zb \leq \Eb\xb + \db, 
\end{equation}
which arise, for instance, in MPC. In~\eqref{eq:mpQP_primal}, $\zb \in \R^l$ denotes the decision variable and $\xb\in \R^n$ is a parameter (e.g., the system state in MPC). The matrix $\Hb \in \R^{l \times l}$ is assumed to be positive definite, and $\Gb \in \R^{q \times l}$ is associated with $q \in \N$ constraints. We denote the set of parameters $\xb$, for which \eqref{eq:mpQP_primal} is feasible, with $\Fc$. 
For a given $\xb \in \Fc$, we seek to find (an approximation of) the optimizer $\zb^\ast(\xb)$.
One of the simplest schemes to obtain a numerical solution are projected gradient descent schemes. To avoid numerically demanding projections onto the primal constraint set $\Zc:=\{ \zb \in \R^l \,|\, \Gb \zb \leq \Eb\xb + \db \}$, one often considers the dual problem of \eqref{eq:mpQP_primal}, which can be formulated~as
\begin{equation}
\label{eq:mpQP_dual}
\min_{\mub \geq \zerob}  \,\frac{1}{2}  \mub^\top \Gb \Hb^{-1} \Gb^\top \mub +  \mub^\top ( \Gb \Hb^{-1}\Fb\xb  + \Eb\xb + \db) 
\end{equation}
after eliminating constant terms in the objective function. Notably, the constraint set $\Pc:=\{ \mub \in \R^q \,|\, \mub \geq \zerob \}$ in \eqref{eq:mpQP_dual} always refers to the positive orthant in $\R^q$.

\subsection{Dual projected gradient descent solver}\label{sec:Preliminaries:DPGD}

Since the projection of a $\mub \in \R^q$ onto $\Pc$ 
can be realized via $\max \{ \mub , \zerob \}$, implementing a projected gradient descent scheme for solving the dual~\eqref{eq:mpQP_dual} is straightforward. To formalize this, we introduce the matrices $\Mb:=\Gb \Hb^{-1} \Gb^\top$ and $\Nb:=\Gb \Hb^{-1} \Fb+\Eb$, and note that $\Mb \mub + \Nb \xb + \db$ refers to the gradient of the objective function in~\eqref{eq:mpQP_dual}. Hence, a dual projected gradient descent scheme is realized by the iterations
\begin{equation}
\label{eq:solver_iteration}
\mub^{i+1}:= \max \left\{\mub^{i} - \alpha_i \left( \Mb \mub^{i} + \Nb \xb + \db \right), \zerob \right\}, 
\end{equation}
where $\alpha_i$ denotes the step size during iteration $i\in \N$. Remarkably, the iterates $\mub^{i+1}$ depend on the given parameter~$\xb$, but we omit explicitly indicating this dependence throughout the paper for the sake of notational simplicity. Now, the scheme is typically initialized with $\mub^0:=\zerob$, exploiting the fact that Lagrange multipliers associated with inactive (or weakly active) constraints are equal to zero. Moreover, in its simplest form, the step size $\alpha_i$ is kept constant throughout the iterations. The standard (and optimal) choice then is $\alpha_i=\alpha_{\text{std}}:=1/\lambda_{\max}(\Mb)$ \cite{Boyd2004}, where $\lambda_{\max}(\Mb)$ denotes the largest eigenvalue of the positive semi-definite matrix $\Mb$. In fact, this choice guarantees convergence to an\footnote{Notably, while $\zb^\ast(\xb)$ is unique, there may exist multiple $\mub^\ast(\xb)$ for some $\xb$. We comment on this special (and rare) case where needed.} optimizer $\mub^\ast(\xb)$ for every $\xb \in \Fc$.
We intend to provide improved step sizes $\alpha_i(\xb)$ with the help of tailored NNs.

\subsection{KKT conditions and optimal solutions}

In order to specify tailored NNs, it turns out to be helpful to investigate optimal solutions to \eqref{eq:mpQP_primal} and \eqref{eq:mpQP_dual}. Clearly, these are characterized by the well-known Karush-Kuhn-Tucker (KKT) conditions 
\begin{subequations}
\label{eq:KKT}
\begin{align}
    \Hb \zb^\ast(\xb) + \Fb \xb +\Gb^\top \mub^\ast(\xb) & = \zerob,\\
    \Gb \zb^\ast(\xb) - \Eb\xb - \db &\leq \zerob, \\
    \mub^\ast(\xb) & \geq \zerob, \\
    \mathrm{diag}\left(\mub^\ast(\xb) \right) (  \Gb \zb^\ast(\xb) - \Eb\xb - \db) &= \zerob,
\end{align}
\end{subequations}
where $\mathrm{diag}(\mub)$ refers to a diagonal matrix with the elements of $\mub$ on its diagonal.
A special solution arises for $\xb$, for which the ``unconstrained solution'' $\zb^\ast(\xb)=-\Hb^{-1} \Fb \xb$ satisfies the constraint in  \eqref{eq:mpQP_primal}. This applies for all $\xb$ that satisfy $\Gb\zb^\ast(\xb)\leq\Eb\xb+\db$, i.e., for all $\xb$ in the set
$$
\Rc_0 := \{ \xb \in \R^n \,|\, \zerob \leq \Nb x + \db \}.
$$
The special solution is completed by $\mub^\ast(\xb)=\zerob$ as easily confirmed by the KKT conditions.
For the remaining $\xb \in {\Fc \setminus \Rc_0}$, optimal solutions are characterized by (strongly) active constraints.
In fact, for every $\xb \in {\Fc \setminus \Rc_0}$ and related optimizer $\mub^\ast(\xb)$, there exists a set\footnote{Notably, in case there exist multiple $\mub^\ast(\xb)$, there exists a set $\Sc$ satisfying \eqref{eq:conditionsOnS} and \eqref{eq:zLambdaOptS} for every optimizer $\mub^\ast(\xb)$.}
\begin{equation}
\label{eq:conditionsOnS}
    \Sc \subseteq \Qc:=\{1,\dots,q\} \quad \text{with} \quad 1 \leq |\Sc| = \mathrm{rank}(\Gb_\Sc)
\end{equation}
such that 
\begin{equation}
\label{eq:zLambdaOptS}
 \begin{pmatrix}
    \zb^\ast(\xb) \\
    \mub_\Sc^\ast(\xb)
\end{pmatrix}= \begin{pmatrix}
    \Hb & \Gb_\Sc^\top \\
    \Gb_\Sc & \zerob
\end{pmatrix}^{-1} \begin{pmatrix}
    -\Fb \xb \\
    \Eb_\Sc \xb + \db_\Sc 
\end{pmatrix}
\end{equation}
and $\mub^\ast_{\Qc \setminus \Sc}(\xb)=\zerob$, where $\Gb_{\Sc}$ (and analogously $\Eb_\Sc$ etc.) refers to the submatrix of $\Gb$ resulting from deleting those rows of $\Gb$ that are not contained in $\Sc$. Furthermore, $\Gb_\Sc^\top$ is short-hand notation for $(\Gb_\Sc)^\top$. Remarkably, due to $\textrm{rank}(\Gb_\Sc)=|\Sc|$, the inverse in~\eqref{eq:zLambdaOptS} exists and evaluates to
$$
\begin{pmatrix}
    \Hb^{-1} + \Hb^{-1} \Gb_\Sc^\top  \Mb_{\Sc,\Sc}^{-1} \Gb_\Sc \Hb^{-1} & \Hb^{-1} \Gb_\Sc^\top  \Mb_{\Sc,\Sc}^{-1} \\
    \Mb_{\Sc,\Sc}^{-1}  \Gb_\Sc \Hb^{-1} & -\Mb_{\Sc,\Sc}^{-1} 
\end{pmatrix},
$$
where $\Mb_{\Sc_1,\Sc_2}$ refers to the submatrix of $\Mb$ resulting from deleting those rows respectively columns of $\Mb$ that are not contained in $\Sc_1$ respectively $\Sc_2$. Further, $ \Mb_{\Sc,\Sc}^{-1}$ is short-hand notation for $(\Mb_{\Sc,\Sc})^{-1}$.
Hence, we find 
\begin{subequations}
\label{eq:optimizers}
\begin{align}
\label{eq:lambdaOpt}
 \mub_\Sc^\ast(\xb) &=-\Mb_{\Sc,\Sc}^{-1} \left( \Nb_\Sc  \xb + \db_\Sc \right) \qquad \text{and} \\  
 \label{eq:zOpt}
 \zb^\ast(\xb)&=-\Hb^{-1} \big(\Fb \xb + \Gb_\Sc^\top \mub_\Sc^\ast(\xb)\big).
\end{align}
\end{subequations}
Remarkably, we can also reverse the perspective and, for a given set $\Sc$ as in~\eqref{eq:conditionsOnS}, investigate for which $\xb$ the right-hand sides in~\eqref{eq:optimizers} reflect optimizers. To formalize this approach, we introduce the solution candidates
\begin{equation}
\label{eq:muOptHat}
    \hat{\mub}^\ast_\Sc(\xb) :=-\Mb_{\Sc,\Sc}^{-1} \left( \Nb_\Sc  \xb + \db_\Sc \right)
\,\,\,\text{with} \,\,\, \hat{\mub}^\ast_{\Sc^c}(\xb) :=\zerob
\end{equation} 
and, analogously, $\hat{\zb}^\ast(\xb)$. 
According to the KKT conditions \eqref{eq:KKT}, the specified solution candidates are optimizers if $\hat{\mub}^\ast_\Sc(\xb) \geq \zerob$ and $\Gb_{\Sc^c} \hat{\zb}^\ast(\xb) \leq \Eb_{\Sc^c} \xb + \db_{\Sc^c}$, where $\Sc^c:=\Qc \setminus \Sc$. Hence, a set $\Sc$ applies for all $\xb$ in
\begin{align}
\nonumber
\!\Rc_{\Sc}:= \{ \xb \in \R^n \,|\,  &\hat{\mub}^\ast_\Sc(\xb) \geq \zerob, \\
\label{eq:setRS}
& \Mb_{\Sc^c\!,\Sc} \, \hat{\mub}^\ast_\Sc(\xb) + \Nb_{\Sc^c} \xb + \db_{\Sc^c} \geq \zerob \},\!
\end{align}
where we note that $\Rc_\Sc$ may be empty for some $\Sc$. Notably, the union of all such $\Rc_\Sc$ and $\Rc_0$ is identical to $\Fc$ \cite{Bemporad2002}.

\subsection{Neural networks}

We aim to learn parameter-dependent step sizes $\alpha_i(\xb)$ for \eqref{eq:solver_iteration} by deep feed-forward-NNs, i.e., we consider NNs of the form $\Phib(\xib)=\fb^{l+1}\circ \gb^{l}\circ \fb^{l}\circ \dots \circ \gb^{1}\circ \fb^{1}(\xib)$ with $l \in\N$ hidden layers and $w_i$ neurons in layer $i$. The functions $\fb^{i}: \R^{w_{i-1}} \rightarrow \R^{w_i}$ for $i \in \{1,\dots,l\}$ refer to preactivations and $\gb^{i}: \R^{p_i w_i} \rightarrow \R^{w_i}$ to activation functions. Finally, $\fb^{l+1}: \R^{w_{l}} \rightarrow \R^{w_{l+1}}$ reflects a postactivation. The functions $\fb^{i}$ are typically affine, i.e., $\fb^{i}(\yb^{i-1})=\Wb^{i}\yb^{i-1}+\vb^{i}$, where $\Wb^{i}\in\R^{w_i\times w_{i-1}}$ is a weighting matrix, $\vb^{i} \in \R^{w_i}$ is a bias vector, and $\yb^{i-1}$ denotes the output of the previous layer with $\yb^{0}:=\xib$. We will summarize the parameters of an NN by $\thetab$ and occasionally use the notation $\Phib(\xib,\thetab)$ when we explicitly refer to the dependence of the NN on the parameters. The activation function is typically nonlinear, where common choices are sigmoid, hyperbolic tangent, and rectified linear unit (ReLU). Due to reasons we discuss in the sections \ref{sec:opt_steps} and \ref{sec:validityAssumptions} we will focus here on the sigmoid and ReLU activation. The sigmoid activation function is defined as
\begin{equation}
    \label{eq:sig}
    \gb_{\text{sig}}^{i}(\zb^{i})=\sigma\left(\zb^{i}\right):=\begin{pmatrix}
        \frac{1}{1+e^{-\zb_1^{i}}} \\
        \vdots \\
        \frac{1}{1+e^{-\zb_{w_i}^{i}}}
    \end{pmatrix}
\end{equation}
and the ReLU activation function as
\begin{equation}
    \label{eq:ReLU}
    \gb_{\text{relu}}^{i}(\zb^{i})=\max\left\{\zerob,\zb^{i}\right\}:=\begin{pmatrix}
        \max\big\{0,\zb_1^{i}\big\} \\
        \vdots \\
        \max\big\{0,\zb_{w_i}^{i}\big\}
    \end{pmatrix}.
\end{equation}
We will refer to the resulting NNs as a sigmoid NN and ReLU NN, respectively.   

\section{Optimal parameter-dependent step sizes}\label{sec:opt_steps}

We next investigate optimal step sizes $\alpha_i$ depending on the parameter $\xb$ and aim to identify a pattern, where we assume $\mub^0:=\zerob$ throughout the remaining paper. We start by noting that, for every $\xb \in \Rc_0$, any choice of $\alpha_i>0$ leads to optimal iterates $\mub^i = \zerob$. In fact, given $\mub^i=\zerob$, the iterations~\eqref{eq:solver_iteration} simplify to
\begin{equation}
    \mub^{i+1} =  \max \left\{ - \alpha_i \left(  \Nb \xb + \db \right), \zerob \right\} = \zerob 
\end{equation}
due to $\Nb \xb + \db \geq \zerob$. The claim then follows by induction, starting with $\mub^0:=\zerob$.

Next, we investigate parameters $\xb$ in the sets $\Rc_\Sc$ and begin by making the following assumptions, which will be justified in Section~\ref{sec:validityAssumptions}.

\begin{assum}
\label{assum:iteration_iStar}
    Consider a set $\Sc$ as in~\eqref{eq:conditionsOnS} and assume that the corresponding set $\Rc_\Sc$ as in \eqref{eq:setRS} is non-empty. Consider any $\xb \in \Rc_\Sc$ and assume that, after $i^\ast>0$ iterations~\eqref{eq:solver_iteration} with the step size $\alpha_{\text{std}}$, the iterate $\mub^{i^\ast}$ satisfies $\mub^{i^\ast}_{\Sc^c} = \zerob$. 
\end{assum}

\begin{assum}
\label{assum:lastSteps}
Let Assumption~1 hold and additionally suppose that the $|\Sc|$ iterations~\eqref{eq:solver_iteration} following iteration $i^\ast$ satisfy $\mub^{i}_{\Sc^c} = \zerob$ and $\mub^{i}_{\Sc} \geq \zerob$ for all $i \in {i^\ast+1,\dots,i^\ast+|\Sc|}$.
\end{assum}

Based on these assumptions, we can specify the iterations~\eqref{eq:solver_iteration} and deduce the following results.

\begin{lem}
\label{lem:simplifiedIterations}
Let Assumption~\ref{assum:lastSteps} hold. Then, the iterates $\mub^i$ satisfy
\begin{equation}
\label{eq:lambdaSi}
 \mub_\Sc^{i+1}  = \left(\Ib - \alpha_i \Mb_{\Sc,\Sc} \right)\mub_\Sc^{i} - \alpha_i \left( \Nb_\Sc \xb + \db_\Sc \right)  
\end{equation}
for every $i \in \{i^\ast,\dots,i^\ast+|\Sc|-1\}$.
\end{lem}
\begin{proof}
The claim follows directly from Assumption~\ref{assum:lastSteps} in combination with the iteration scheme in~\eqref{eq:solver_iteration}.
\end{proof}

\begin{lem}
Let Assumption~\ref{assum:lastSteps} hold. Then, we have
\begin{equation}
\label{eq:lambdaSiLambdaSOpt}
\mub^{i+1}_\Sc-\hat{\mub}_\Sc^\ast(\xb) =\left(\Ib-\alpha_i \Mb_{\Sc,\Sc}\right)(\mub_\Sc^{i}-\hat{\mub}_\Sc^\ast(\xb))
\end{equation}
for every $i \in \{i^\ast,\dots,i^\ast+|\Sc|-1\}$.
\end{lem}
\begin{proof}
The claim follows from~\eqref{eq:lambdaSi} in combination with \eqref{eq:lambdaOpt}.
\end{proof}

\begin{lem}
\label{lem:muiMuOpt}
Let Assumption~\ref{assum:lastSteps} hold. Then, we have
\begin{equation}
\label{eq:lambdaSjLambdaSOptExplicit}
\mub^{i+1}_\Sc-\hat{\mub}_\Sc^\ast(\xb) =\!\left(\prod\limits_{j=i^\ast}^{i}\!\!\left(\Ib-\alpha_j \Mb_{\Sc,\Sc}\right)\!\right)\! (\mub_\Sc^{i^\ast}-\hat{\mub}_\Sc^\ast(\xb))
\end{equation}
for every $i \in \{i^\ast,\dots,i^\ast+|\Sc|-1\}$.
\end{lem}
\begin{proof}
The claim follows by recursively using~\eqref{eq:lambdaSiLambdaSOpt}.
\end{proof}

In general, we have $\mub_\Sc^{i^\ast} \neq \hat{\mub}_\Sc^\ast(\xb)$. Therefore, in order to obtain the desired result $\mub_{\Sc}^{i^\ast+|\Sc|} = \hat{\mub}_\Sc^\ast(\xb)$, evaluating \eqref{eq:lambdaSjLambdaSOptExplicit} at $i := i^\ast + |\Sc| - 1$ shows that
\begin{equation}
\label{eq:thisShouldHold}
\prod\limits_{j=i^\ast}^{i^\ast+|\Sc|-1}\left(\Ib - \alpha_j \Mb_{\Sc,\Sc}\right) = \zerob
\end{equation}
must hold.
Fortunately, this can be readily achieved by appropriately choosing the step sizes $\alpha_i$ as formalized next.

\begin{lem}
\label{lem:choiceOfAlpha}
Let Assumption~\ref{assum:lastSteps} hold and let $\lambda_1,\dots,\lambda_{|\Sc|}$ denote the $|\Sc|$ eigenvalues of $\Mb_{\Sc,\Sc}$ in descending order. 
Then, \eqref{eq:thisShouldHold} holds for the choice
    \begin{equation}
\label{eq:alphaLambda}
\alpha_i:=\frac{1}{\lambda_{i-i^\ast+1}}
\end{equation}
for every $i \in \{i^\ast,\dots,i^\ast+|\Sc|-1\}$. 
\end{lem}

\begin{proof}
    The characteristic polynomial of $\Mb_{\Sc,\Sc}$ evaluates to
    $\prod_{i=1}^{|\Sc|} (\lambda - \lambda_i)$. According to the Cayley-Hamilton theorem, a matrix satisfies its own characteristic polynomial. Hence, we have
    \begin{equation}
\label{eq:CayleyHamilton}
 \prod_{i=1}^{|\Sc|} (\Mb_{\Sc,\Sc} - \lambda_i \Ib) = \zerob.
\end{equation}
Reindexing the product in~\eqref{eq:CayleyHamilton} according to~\eqref{eq:thisShouldHold} and multiplying the resulting equation with the factor 
$\prod_{i=i^\ast}^{i^\ast+|\Sc|-1} \alpha_i$, where the $\alpha_i$ are as in~\eqref{eq:alphaLambda}, yields \eqref{eq:thisShouldHold}. Notably, the applied factor is positive (and real) since each $\lambda_i$ is positive (and real) due to positive definiteness of $\Mb_{\Sc,\Sc}$ ensured by $\mathrm{rank}(\Gb_\Sc)=|\Sc|$.
\end{proof}

The combination of the previous results allows us to show that, under Assumption~\ref{assum:lastSteps}, we can exactly solve \eqref{eq:mpQP_dual} with the iterations \eqref{eq:solver_iteration} given tailored step sizes.

\begin{thm}\label{thm:exact_solution}
  Let Assumption~\ref{assum:lastSteps} hold. 
  Then, evaluating the $|\Sc|$ iterations~\eqref{eq:solver_iteration} following iteration $i^\ast$ with  
 step sizes as in \eqref{eq:alphaLambda} results in $\mub^{i^\ast+|\Sc|}=\hat{\mub}^\ast(\xb)$.
\end{thm}

\begin{proof}
    The statement follows from the combination of Lemmas \ref{lem:simplifiedIterations}, \ref{lem:muiMuOpt}, and \ref{lem:choiceOfAlpha}.
\end{proof}

Notably, the tailored step sizes are typically larger than the standard step size $\alpha_{\text{std}}$.

\begin{cor}\label{cor:optStepsLarger}
 Let Assumption~\ref{assum:lastSteps} hold and let the step sizes $\alpha_i$ be as in Lemma~\ref{lem:choiceOfAlpha}. Then, $\alpha_{\text{std}}\leq \alpha_{i^\ast} \leq \dots \leq \alpha_{i^\ast+|\Sc|-1}$. 
\end{cor}

\section{Validity of assumptions}
\label{sec:validityAssumptions}

The derived step sizes in Section~\ref{sec:opt_steps} depend crucially on Assumption~\ref{assum:lastSteps}. Hence, its validity and applicability must be carefully assessed. Remarkably, if Assumption~\ref{assum:iteration_iStar} holds, the remaining conditions in Assumption~\ref{assum:lastSteps} can be compactly expressed in terms of $\mub^{i^\ast}_{\Sc}$.

\begin{lem}
\label{lem:conditionsForAssum2}
    Let Assumption~\ref{assum:iteration_iStar} hold. Then, Assumption~\ref{assum:lastSteps} holds if and only if $\mub^{i^\ast}_{\Sc}$ is such that
\begin{equation}
\label{eq:conditonsOnMuSiI}
        \left(\prod\limits_{j=i^\ast}^{i}\!\!\left(\Ib-\alpha_j \Mb_{\Sc,\Sc}\right)\!\right)\! (\mub_\Sc^{i^\ast}-\hat{\mub}_\Sc^\ast(\xb)) +\hat{\mub}_\Sc^\ast(\xb) \geq \zerob
\end{equation}
    and
    \begin{align}
    \nonumber
 \Mb_{{\Sc^c}\!,\Sc} \left(\prod\limits_{j=i^\ast}^{i-1}\!\!\!\!\!\left(\Ib-\alpha_j \Mb_{\Sc,\Sc}\right)\!\right)\!  (\mub_\Sc^{i^\ast}-\hat{\mub}_\Sc^\ast(\xb)) & \\
 \label{eq:conditonsOnMuSiII}
 + \Mb_{{\Sc^c}\!,\Sc} \hat{\mub}_\Sc^\ast(\xb) + \Nb_{\Sc^c} \xb + \db_{\Sc^c} & \geq \zerob
\end{align}
    for all $i \in \{i^\ast,\dots,i^\ast+|\Sc|-1\}$.
\end{lem}
\begin{proof}
 The conditions in~\eqref{eq:conditonsOnMuSiI} correspond to the conditions $\mub_\Sc^i \geq \zerob$ from Assumption~\ref{assum:lastSteps}. Noting this, \eqref{eq:conditonsOnMuSiI} then follows from~\eqref{eq:lambdaSjLambdaSOptExplicit}. Analogously, the conditions~\eqref{eq:conditonsOnMuSiII} shall ensure $\mub_{\Sc^c}^i=\zerob$. Assuming $\mub_{\Sc^c}^{i-1}=\zerob$ is given,  \eqref{eq:solver_iteration} implies that $\mub_{\Sc^c}^i=\zerob$ results whenever $\Mb_{\Sc^c,\Sc} \mub_\Sc^{i} + \Nb_{\Sc^c} \xb + \db_{\Sc^c} \geq \zerob$. Again, substituting $\mub_\Sc^i$ via \eqref{eq:lambdaSjLambdaSOptExplicit} yields \eqref{eq:conditonsOnMuSiII}. The claim then follows by induction with the seed $\mub_{\Sc^c}^{i^\ast}=\zerob$ from Assumption~\ref{assum:iteration_iStar}.
\end{proof} 
By using the KKT conditions \eqref{eq:KKT} we can
verify that \eqref{eq:conditonsOnMuSiI} and \eqref{eq:conditonsOnMuSiII} are satisfied for $\mub_\Sc^{i^\ast} = \hat{\mub}_\Sc^\ast(\xb)$. Consequently, given that $\mub_{\Sc^c}^{i^\ast} = \zerob$ by Assumption~\ref{assum:iteration_iStar}, Lemma~\ref{lem:conditionsForAssum2} establishes that Assumption~\ref{assum:lastSteps} holds for $\mub^{i^\ast} = \hat{\mub}^\ast(\xb)$. However, this relation is generally unlikely to hold.
Nevertheless, it is known that $\lim_{i \rightarrow \infty} \mub^{i}$ converges to an optimizer $\mub^\ast(\xb)$ for the standard choice $\alpha_i := \alpha_{\text{std}}$ (see, e.g. \cite{Boyd2004}). At this point, we recall a well-known fact on the uniqueness of $\mub^\ast(\xb)$.

\begin{cor}[cf.~{\cite[Sect.~4.1.1]{Bemporad2002}}]
    Consider a set $\Sc$ as in~\eqref{eq:conditionsOnS} and assume that the corresponding set $\Rc_\Sc$ as in \eqref{eq:setRS} is 
full-dimensional. Then, for every $\xb \in \mathrm{int}(\Rc_\Sc)$, the optimizer $\mub^\ast(\xb)$ is unique and equal to $\hat{\mub}^\ast(\xb)$.
\end{cor}

As a consequence, for every $\xb \in \mathrm{int}(\Rc_\Sc)$,  $\| \mub^{i^\ast} - \hat{\mub}^\ast(\xb) \|_2$ can be made arbitrarily small by choosing  $i^\ast$ sufficiently large. 
Moreover, for every $\xb \in \mathrm{int}(\Rc_\Sc)$, we obviously have $\hat{\mub}_\Sc^\ast(\xb)>\zerob$ and $\Mb_{\Sc^c\!,\Sc} \, \hat{\mub}_\Sc^\ast(\xb) + \Nb_{\Sc^c} \xb + \db_{\Sc^c} > \zerob $ by construction of $\Rc_\Sc$. 
Combining these observations, it follows that for every $\xb \in \mathrm{int}(\Rc_\Sc)$, the conditions \eqref{eq:conditonsOnMuSiI} and \eqref{eq:conditonsOnMuSiII} are satisfied for all sufficiently large  $i^\ast$.  

Unfortunately, we still need to satisfy $ \mub_{\Sc^c}^{i^\ast}= \zerob$ from Assumption~\ref{assum:iteration_iStar}, and this equality condition does not allow for any deviation. However, it will turn out that for the preceding iterate $\mub^{i^\ast-1}$, we may have $\mub_{\Sc^c}^{i^\ast-1} \neq \zerob$. In fact, according to~\eqref{eq:solver_iteration}, we obtain $\mub_{\Sc^c}^{i^\ast} = \zerob$ whenever 
\begin{equation}
    \label{eq:iMinusScCondition}
    \mub_{\Sc^c}^{i^\ast-1} - \alpha_{\text{std}} \big( \Mb_{\Sc^c} \mub^{i^\ast-1} + \Nb_{\Sc^c} \xb + \db_{\Sc^c} \big) \leq \zerob. 
\end{equation}
If we further restrict ourselves to $\mub^{i^\ast-1}$ yielding $\mub_{\Sc}^{i^\ast} \geq \zerob$, the additional condition
\begin{equation}
    \label{eq:iMinusS_Condition}
\mub_{\Sc}^{i^\ast-1} - \alpha_{\text{std}} \big( \Mb_{\Sc} \mub^{i^\ast-1}  + \Nb_\Sc \xb + \db_\Sc \big)  \geq \zerob
\end{equation}
results. Slightly reformulating the conditions results in the following statements.

\begin{lem}
\label{lem:Assum1holds}
    Consider a set $\Sc$ as in~\eqref{eq:conditionsOnS} and assume that the corresponding set $\Rc_\Sc$ as in \eqref{eq:setRS} is non-empty. Consider any $\xb \in \Rc_\Sc$ and any $i^\ast>0$. Then, Assumption~\ref{assum:iteration_iStar} holds iff 
   \begin{align}
   \nonumber
    \Big( \Mb_{\Sc^c} - \nicefrac{1}{\alpha_{\text{std}}}  \begin{pmatrix}
 \zerob & \Ib 
\end{pmatrix} \!\Big) (\mub^{i^\ast-1} -\hat{\mub}^\ast(\xb) ) &  \\
\label{eq:conditionAssum1}
    + \Mb_{\Sc^c,\Sc} \hat{\mub}_\Sc^\ast(\xb) +  \Nb_{\Sc^c} \xb + \db_{\Sc^c} &\geq \zerob. 
\end{align} 
Moreover, the relation 
\begin{equation}
\label{eq:toBeShown}
 \mub_\Sc^{i^\ast} -\hat{\mub}_\Sc^\ast(\xb)  =  \Big(\!\begin{pmatrix}
    \Ib & \zerob
\end{pmatrix} - \alpha_{\text{std}} \Mb_{\Sc} \Big)  \big( \mub^{i^\ast-1}  -  \hat{\mub}^\ast(\xb)  \big)   
\end{equation}
holds if and only if
\begin{equation}
\label{eq:toBeShown2}
\Big(\!\begin{pmatrix}
    \Ib & \zerob
\end{pmatrix} - \alpha_{\text{std}} \Mb_{\Sc} \Big)  \big( \mub^{i^\ast-1}  -  \hat{\mub}^\ast(\xb)  \big) + \hat{\mub}_\Sc^\ast(\xb) \geq \zerob.
\end{equation}
\end{lem}

\begin{proof}
We first note that~\eqref{eq:iMinusScCondition} can be rewritten as
$$
\Mb_{\Sc^c} \big( \mub^{i^\ast\!-1} \!- \hat{\mub}^\ast(\xb) \big) + \Mb_{\Sc^c\!,\Sc} \hat{\mub}_\Sc^\ast(\xb)+ \Nb_{\Sc^c} \xb + \db_{\Sc^c} \!\geq \frac{\mub_{\Sc^c}^{i^\ast-1}}{\alpha_{\text{std}}}\!. 
$$
In fact, $\Mb_{\Sc^c,\Sc} \hat{\mub}_\Sc^\ast(\xb)-\Mb_{\Sc^c} \hat{\mub}^\ast(\xb) =\zerob$ due to $\hat{\mub}_{\Sc^c}^\ast(\xb)=\zerob$.
The final reformulation yielding~\eqref{eq:conditionAssum1} follows from basic manipulations. Regarding the second statement, we first note that satisfying~\eqref{eq:iMinusS_Condition} implies that the left-hand side of \eqref{eq:iMinusS_Condition} equals $\mub_\Sc^{i^\ast}$. Hence, given \eqref{eq:iMinusS_Condition}, we derive
\begin{equation}
\label{eq:almostThere}
   \mub_\Sc^{i^\ast}  =  \mub_\Sc^{i^\ast-1} - \alpha_{\text{std}} \Mb_{\Sc}  \big( \mub^{i^\ast-1}  -  \hat{\mub}^\ast(\xb)  \big) 
\end{equation}
by exploiting $-\Mb_{\Sc} \hat{\mub}_\Sc^\ast(\xb) = -\Mb_{\Sc,\Sc} \hat{\mub}_\Sc^\ast(\xb) = \Nb_\Sc \xb + \db_\Sc $. Now, \eqref{eq:almostThere} can be turned into \eqref{eq:toBeShown} by subtracting $\mub_\Sc^\ast(\xb)$ on both sides and using basic manipulations. Using a similar approach, we can finally show that the right-hand side of \eqref{eq:almostThere} is equivalent to~\eqref{eq:toBeShown2}.
\end{proof}

Obviously, the conditions~\eqref{eq:conditionAssum1} and \eqref{eq:toBeShown2} have a structure similar to that of \eqref{eq:conditonsOnMuSiII} and  \eqref{eq:conditonsOnMuSiI}, respectively. As a consequence, conditions~\eqref{eq:conditionAssum1} and \eqref{eq:toBeShown2} can be satisfied for every $\xb \in \mathrm{int}(\Rc_\Sc)$ by choosing $i^\ast$ sufficiently large, based on the same arguments as before. Given \eqref{eq:toBeShown2}, \eqref{eq:toBeShown} implies that small $\|\mub^{i^\ast-1} -\hat{\mub}^\ast(\xb) \|_2$ translate to small $\|\mub^{i^\ast} -\hat{\mub}^\ast(\xb) \|_2$. As a consequence, for every $\xb \in \mathrm{int}(\Rc_\Sc)$, the conditions \eqref{eq:conditonsOnMuSiI}, \eqref{eq:conditonsOnMuSiII}, \eqref{eq:conditionAssum1} and \eqref{eq:toBeShown2} can be satisfied simultaneously for sufficiently large $i^\ast$.  This allows us to state our central result.

\begin{thm}
\label{thm:centralResult}
   Consider a set $\Sc$ as in~\eqref{eq:conditionsOnS} and assume that the corresponding set $\Rc_\Sc$ in \eqref{eq:setRS} is 
full-dimensional. Then, for every $\xb \in \mathrm{int}(\Rc_\Sc)$, there exists an  $i^\ast>0$ such that $i^\ast$ iterations~\eqref{eq:solver_iteration} with the step size $\alpha_{\text{std}}$ followed by $|\Sc|$ iterations~\eqref{eq:solver_iteration} with step sizes as in \eqref{eq:alphaLambda} yield ${\mub^{i^\ast+|\Sc|}=\mub^\ast(\xb)}$.
\end{thm}

\begin{proof}
For every $\xb \in \mathrm{int}(\Rc_\Sc)$, there exists an $\epsilon>0$ such that $\hat{\mub}_\Sc^\ast(\xb) \geq \epsilon \oneb$ and $\Mb_{{\Sc^c}\!,\Sc} \hat{\mub}_\Sc^\ast(\xb) + \Nb_{\Sc^c} \xb + \db_{\Sc^c}  \geq \epsilon \oneb$, where $\oneb$ denotes the vector of all ones. Furthermore, there exists a $\delta>0$ such that, for every $\Delta \mub \in \R^q$ with $\| \Delta \mub \|_2\leq \delta$, we have  
\begin{equation}
\label{eq:firstBlock}
\!\left\|  \left(\prod\limits_{j=i^\ast}^{i}\!\!\left(\Ib-\alpha_j \Mb_{\Sc,\Sc}\right)\!\right)\!\! \Big(\!\!\begin{pmatrix}
    \Ib & \!\zerob
\end{pmatrix} \!- \alpha_{\text{std}} \Mb_{\Sc} \Big)  \Delta \mub \right\|_\infty \!\!\!\!\!\leq \epsilon \!   
\end{equation}
for all $i \in \{i^\ast-1,\dots,i^\ast+|\Sc|-1\}$,  
$$
\left\| \Mb_{{\Sc^c}\!,\Sc} \!\left(\prod\limits_{j=i^\ast}^{i-1}\!\!\left(\Ib\!-\alpha_j \Mb_{\Sc,\Sc}\right)\!\!\right)\! \! \Big(\!\!\begin{pmatrix}
    \Ib & \!\zerob
\end{pmatrix} \! - \alpha_{\text{std}} \Mb_{\Sc} \Big)  \Delta \mub \right\|_\infty \!\!\!\!\!\!\leq \epsilon
$$
for all $i \in \{i^\ast,\dots,i^\ast+|\Sc|-1\}$, and
$$
\left\| \Big( \Mb_{\Sc^c} - \nicefrac{1}{\alpha_{\text{std}}}  \begin{pmatrix}
 \zerob & \Ib 
\end{pmatrix} \!\Big) \Delta \mub \right\|_\infty \leq \epsilon.
$$
In combination, we can ensure that the conditions \eqref{eq:conditonsOnMuSiI}, \eqref{eq:conditonsOnMuSiII}, \eqref{eq:conditionAssum1} and \eqref{eq:toBeShown2} hold for all $\mub^{i^\ast-1}$ satisfying 
\begin{equation}
\label{eq:conditionOnMu_iMinus1}
   \| \mub^{i^\ast-1} - \hat{\mub}^\ast(\xb) \|_2 \leq \delta, 
\end{equation}
 where we note that \eqref{eq:toBeShown2} is captured in~\eqref{eq:firstBlock} for $i:=i^\ast-1$. As previously discussed, satisfying \eqref{eq:conditionOnMu_iMinus1} is always possible for a sufficiently large choice of $i^\ast$.
Hence, the conditions in Lemmas~\ref{lem:conditionsForAssum2} and \ref{lem:Assum1holds} are satisfied, which implies that Assumption~\ref{assum:lastSteps} holds. As a consequence, Theorem~\ref{thm:exact_solution} applies, which completes the proof.
\end{proof}

We eventually investigate the structure of the optimal step sizes $\alpha_i$ as function of $\xb \in \Rc_\Sc$. To this end, we introduce the function $\varphib: \Fc \rightarrow \R^l$ defined as follows. For every $\xb \in \Fc \setminus \Rc_0$, we identify a set $\Sc$ and optimizer $\mub^\ast(\xb)$ satisfying~\eqref{eq:conditionsOnS} and \eqref{eq:zLambdaOptS}, and let $\lambda_1,\dots,\lambda_{|\Sc|}$ be as in Lemma~\ref{lem:choiceOfAlpha}. For every $\xb \in \Rc_0$, we set $\Sc:=\emptyset$. Then, we specify $\varphib_j(\xb)$ as  
$$
\varphib_j(\xb) := \left\{ \begin{array}{ll}
  1/\lambda_{j-l+|\Sc|}   & \text{if} \quad j > l - |\Sc|,   \\
  \alpha_{\text{std}}   & \text{otherwise}
\end{array}\right.
$$
for every $j \in \{1,\dots,l\}$,
where we note that $|\Sc|\leq l$ is guaranteed due to~\eqref{eq:conditionsOnS}. 
Lemma~\ref{lem:choiceOfAlpha} and Theorem~\ref{thm:centralResult} then imply the following structure.

\begin{cor}\label{cor:alpha_PWC}
The function $\varphib$ is piecewise constant (PWC) and uniquely defined almost everywhere on $\Fc$.
\end{cor}

\begin{proof}
  As briefly mentioned in Section 2,  the union of all possible sets $\Rc_\Sc$ and $\Rc_0$ coincides with $\Fc$. For different index sets $\Sc_1$ and $\Sc_2$, the corresponding sets $\Rc_{\Sc_1}$ and $\Rc_{\Sc_2}$ are known to have pairwise disjoint interiors. Hence, for almost every $\xb \in \Fc$, there exists a unique $\Sc$ such that $\xb \in \Rc_\Sc$ and this set is equal for every $\xb \in \mathrm{int}(\Rc_\Sc)$. The observation establishes the PWC structure of $\varphib$ and implies that $\varphib$ is uniquely defined in the interior of every $\Rc_\Sc$.
\end{proof}

In summary, the obtained results are, to the best of the authors' knowledge, the first to show that projected gradient descent solvers with variable step sizes can solve mpQPs exactly. Moreover, we establish for the first time the functional form of the required optimal step sizes. We will now use the insights from this section to construct NNs tailored to solving mpQPs.

\section{Learning optimal step sizes}\label{sec:NNs_for_DPGD}
\begin{figure*}[t]
    \centering
    \begin{subfigure}[t]{0.29\textwidth}
        \centering
        \caption{NN for approximating the step sizes}
        \vspace{2mm}
        \label{fig:NN_for_alpha}
        \input{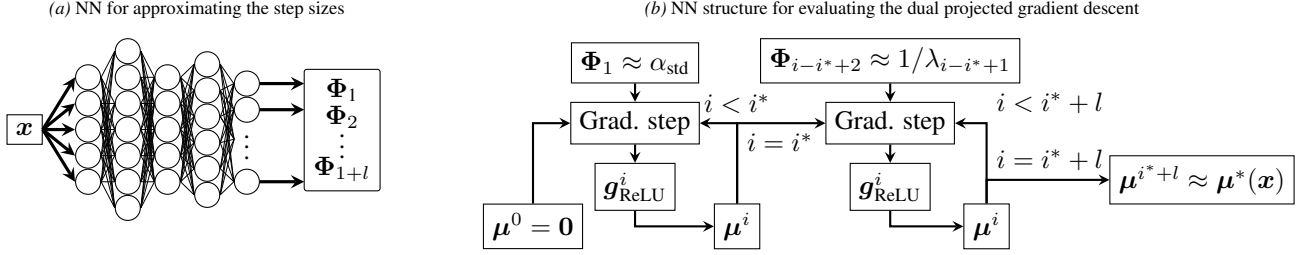}
    \end{subfigure}
    \begin{subfigure}[t]{0.69\textwidth}
        \centering
        \caption{NN structure for evaluating the dual projected gradient descent}
        \vspace{2mm}
        \label{fig:grad_steps}
        \begin{tikzpicture}[scale=1.5,>=latex]
  \newcommand{\lo}{1}
  \newcommand{\loI}{2*\lo}
  \newcommand{\loII}{3*\lo}
  \newcommand{\loi}{0.5*\lo}
  \newcommand{\loii}{0.33*\lo}
  \def\offset{0.2}  
  \newcommand{\cornerarrowhorizontal}[2]{\draw[-stealth,thick](#1)-| ($ (#1)!0.5!(#2) $)|- (#2);}
  \newcommand{\cornerarrowvertical  }[2]{\draw[-stealth,thick](#1)|- ($ (#1)!0.5!(#2) $)-| (#2);}
  \newcommand{\sys}{Grad. step}
  \newcommand{\relu}{$\gb^{i}_{\text{ReLU}}$}
  \begin{scope}[shift= {(0,0)}]
      \node[draw,behind path,anchor=south] (ac)    at (1*\lo ,0.7*\loi) {$\Phib_1\approx\alpha_{\text{std}}$};
      \node[draw,behind path] (sys1)  at (1*\lo  ,0) {\sys};
      \node[draw,behind path,anchor=north] (relu0) at (1*\lo  ,-0.7*\loi) {\relu};
      
      \node[draw,behind path] (sys3)  at (3.25*\lo,0) {\sys};
      \node[draw,behind path,anchor=north] (relu1) at (3.25*\lo  ,-0.7*\loi) {\relu};
      \node[draw,behind path,anchor=south] (ai)    at (3.25*\lo,0.7*\loi) {$\Phib_{i-i^\ast+2}\approx1/\lambda_{i-i^\ast+1}$};
    
      \node[draw, behind path] (lamc0)at (0.1*\lo  ,-0.9\lo) {$\mub^{0}= \zerob$};
      \node[draw, behind path] (lamck)at (1.9*\lo  ,-0.9\lo) {$\mub^{i}$};
      \node[draw, behind path] (lami) at (4.1*\lo  ,-0.9\lo) {$\mub^{i}$};
      \node[draw, behind path] (lamNm)at (6.0*\lo  ,-\loi){$\mub^{i^{\ast}+l}\approx\mub^\ast(\xb)$};
    
  \end{scope}
  
  \draw[-stealth,thick] (ac.south) -- (sys1.north);
  \draw[-stealth,thick] (lamc0.north) |- (sys1.west);
  \draw[-stealth,thick] (sys1.south) -- (relu0.north);
  \draw[-stealth,thick] (relu0.south) |- (lamck.west);
  \draw[-stealth,thick] (lamck.north) |- (sys1.east) node[midway,anchor=south]{$i<i^{\ast}$};
  \draw[-stealth,thick] (lamck.north) |- (sys3.west) node[midway,anchor=north west]{$i=i^{\ast}$};
  \draw[-stealth,thick] (ai.south) -- (sys3.north);
  \draw[-stealth,thick] (sys3.south) -- (relu1.north);
  \draw[-stealth,thick] (relu1.south) |- (lami.west);
  \draw[-stealth,thick] (lami.north) |- (sys3.east)node[midway,anchor=south west]{$i< i^{\ast} +l$};
  \draw[-stealth,thick] (lami.north) |- (lamNm.west)node[midway,anchor=south west]{$i= i^{\ast} +l$};  
  
\end{tikzpicture}
    \end{subfigure}
    \caption{Neural network $\Phib(\xb)$ for approximating the optimal step sizes \eqref{eq:alphaLambda} in (a). In (b), we illustrate the implementation of the dual projected gradient descent solver using ReLU activations and the approximated step sizes from $\Phib(\xb)$. The \enquote{Grad. step} boxes stand for an unprojected gradient step $\yb^{i} - \alpha \left( \Mb \yb^{i} + \Nb \xb + \db \right)$ with a step size $\alpha$ given by the output of the NN $\Phib(\xb)$ from (a).
    }
    \label{fig:NNs}
\end{figure*}
In this section, we will introduce and investigate NNs that are suitable for representing a dual projected gradient descent (DPGD) solver, which, according to Theorem \ref{thm:centralResult}, allows an exact computation of the optimizer of \eqref{eq:mpQP_dual}. In order to find such a representation, we need NNs that represent the optimal step sizes \eqref{eq:alphaLambda} which are, according to Corollary~\ref{cor:alpha_PWC}, PWC functions. For learning such functions, NNs with step activations would be the natural choice. However, the step function is rarely used in practice due to problems with backpropagation-based training caused by gradients with a value of zero. Therefore, we will use the sigmoid activation \eqref{eq:sig} as a proxy for the step activation. The NN $\Phib(\xb,\thetab)$ representing the variable step sizes is trained by minimizing the quadratic cost
\begin{equation}\label{eq:NN_cost}
    \frac{1}{D}\sum\limits_{j=1}^{D}\norm{\mub^{\ast}(\xb^{j})-\mub^{k}}_2^2.
\end{equation}
for a data set $\Nc:=\{(\xb^{i},\sb^{i}) \ | \ \sb^{i}=\mub^\ast(\xb^i) \ \forall i \in \{1,\dots,D\} \}$ containing samples of the dual QP \eqref{eq:mpQP_dual}. The iterate $\mub^{k}$ is computed by evaluating \eqref{eq:solver_iteration} until $i=i^\ast+l-1=:k$ with $\mub^0=\zerob$ for $\xb^j$ and step sizes given by
\begin{equation}\label{eq:alpha_NN}
    \alpha_i\!=\!\left\{\!\!\!\begin{array}{ccc}
             \Phib_1(\xb,\thetab) & \text{if} & i \in \{0,\dots,i^\ast-1\}  \\
             \Phib_{i-i^\ast+2}(\xb,\thetab) & \text{if} & i \in \{i^\ast,\dots,i^\ast+l-1\}.
        \end{array}\right.
\end{equation}
From Theorem \ref{thm:centralResult} and Corollary \ref{cor:alpha_PWC}, we know that for such a DPGD solver, there exist PWC step sizes $\alpha_i=\alpha_i^\ast$ and an $i^\ast\in\N$ such that $\mub^{k}(\xb)=\mub^\ast(\xb)$ for $k=i^\ast+|\Sc|$. Thus, for a sufficiently large $i^\ast$ and a sigmoid NN $\Phib(\xb,\thetab)$ as proxy for a NN with step activation, we should be able to get arbitrarily close to the optimum $\mub^\ast (\xb)$ with $\mub^{k}(\xb)$. The topology of the considered sigmoid NN is shown in Figure~\ref{fig:NN_for_alpha}. As apparent from Theorem~\ref{thm:centralResult} for computing the optimizer of the QP \eqref{eq:mpQP_dual}, we need at most $i^\ast+|\Sc|\leq i^\ast+l$ step sizes. However, since the optimal step sizes for $i<i^\ast$ are constant, i.e., $\alpha_i=\alpha_{\text{std}}$ for all $i\in\{0,\dots,i^\ast-1\}$, we only need an output dimension of $1+l$ for the NN from Figure~\ref{fig:NN_for_alpha}.   
\begin{rem}
    The solver iterations \eqref{eq:solver_iteration}, which incorporate variable step sizes determined by the NN $\Phib(\xb)$ as required in \eqref{eq:NN_cost}, can be realized using ReLU activations. The forward pass of this specific architecture is illustrated in Figure~\ref{fig:grad_steps}. This implementation is particularly advantageous as it allows the evaluation of the cost function \eqref{eq:NN_cost} and the computation of its gradients, which are essential for training $\Phib(\xb)$, within established machine learning frameworks, e.g., TensorFlow \cite{tensorflow2015-whitepaper}, PyTorch \cite{PyTorch}, scikit-learn \cite{scikit-learn}.
\end{rem}     

\subsection{Case study: Learning-enhanced solver}\label{sec:NNs_for_DPGD:A}

In this case study, we aim to investigate the performance of a learning-enhanced DPGD solver with a topology as in Figure~\ref{fig:NNs}. For the case study, we consider the dual OP of an OCP as in, e.g., \cite[Eq.~(4)]{Teichrib2023CDC}, with $N=2$, $\Qb=\Ib$, $\Rb=0.1$, $\Xc=\{\xb\in\R^2 \ | \ |\xb_1|\leq25,|\xb_2|\leq5\}$, and $\Uc=\{u\in\R \ | \ |u|<1\}$ for a discrete-time double integrator \cite[Eq. (2.8)]{Gutman1987} described by 
\begin{equation*}
    \xb(k+1)=\begin{pmatrix}
        1 & 1 \\
        0 & 1
    \end{pmatrix}\xb(k)+\begin{pmatrix}
        0.5 \\
        1
    \end{pmatrix}\ub(k).
\end{equation*}
Given the specified parameters of the OCP, we can derive the matrices $\Hb\in\R^{2\times2}$, $\Fb\in\R^{2\times2}$, $\Gb\in\R^{16\times2}$, $\Eb\in\R^{16\times2}$, and $\db\in\R^{16\times1}$ of the corresponding dual QP \eqref{eq:mpQP_dual}. This results here in a dual QP with $\xb\in\R^2$ and $\mub\in\R^{16}$, where we have $\zb\in\R^l$ with $l=2$ for the primal QP. 

\begin{figure}[t]
    \centering
    \begin{subfigure}[ht]{0.45\textwidth}
        \centering
        \begin{tikzpicture}
  \begin{axis}[
    width=\textwidth,
    height = 0.5*\textwidth,
    axis x line=bottom,
    axis y line=left, 
    xlabel= $i^{\ast}$,
    xtick = {2,4,6,8,10,12,14,16,18},
    ylabel= $\|\mub^{\ast}-\mub^{i^{\ast}+l}\|_{2}$,
    xlabel style = {font=\small},
    ylabel style = {font=\small},
    ymax = 149,
    ymin = 0,
    legend style={{at=(1.02,1.02),anchor=north east},font=\tiny}
  ]
    \addplot[no marks,color=orange, very thick] coordinates {
        (0,150.00)
        (1,83.14)
        (2,52.40)
        (3,43.08)
        (4,38.91)
        (5,35.78)
        (6,32.85)
        (7,30.10)
        (8,27.72)
        (9,25.79)
        (10,24.30)
        (11,23.17)
        (12,22.26)
        (13,21.49)
        (14,20.80)
        (15,20.16)
        (16,19.57)
        (17,19.03)
        (18,18.54)
        (19,18.09)};
    \addlegendentry{Optimal steps after $i^{\ast}$}
    \addplot[no marks,color=black,dashed, very thick] coordinates {
        (0,76.96)
        (1,75.24)
        (2,73.68)
        (3,72.24)
        (4,70.88)
        (5,69.59)
        (6,68.36)
        (7,67.18)
        (8,66.04)
        (9,64.95)
        (10,63.90)
        (11,62.88)
        (12,61.89)
        (13,60.94)
        (14,60.01)
        (15,59.10)
        (16,58.22)
        (17,57.37)
        (18,56.53)
        (19,55.71)};
    \addlegendentry{Constant steps after $i^{\ast}$}
    \end{axis}
\end{tikzpicture}
    \end{subfigure}
    \vspace{-6mm}
    \caption{For each $i^\ast$ the figure shows the distance of $\mub^{i^\ast+l}$ to the optimum averaged over 231 randomly chosen $\xb$ that have not been used during training. The iterate $\mub^k$ is computed by the DPGD solver with different step sizes. The dashed line shows the results when using a standard DPGD solver with $\mub^0=\zerob$ and constant step sizes $\alpha_i=\alpha_{\text{std}}$ for all $i\in\{0,\dots,i^\ast+l-1\}$. The solid line shows the results obtained with the learning-enhanced DPGD solver using the learned step sizes \eqref{eq:alpha_NN}.}
    \label{fig:step_evaluations}
\end{figure}
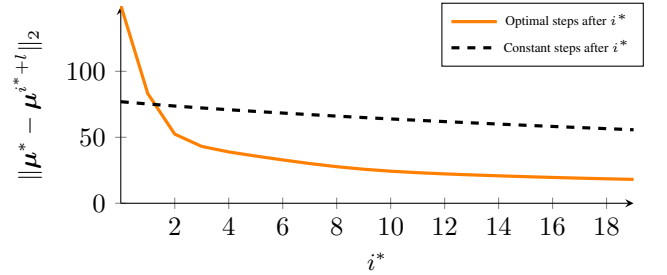

For the NN $\Phib(\xb)$ approximating the $l+1=3$ step sizes, shown in Figure~\ref{fig:NN_for_alpha}, we choose a NN structure with two hidden layers.
First, a sigmoid layer and then a softmax layer before postactivation. This final layer ensures that, before postactivation, only one of the output neurons is approximately $1$ and all others are approximately $0$ for a given input $\xb$.
The actual values of the PWC step sizes are then represented by the parameters of the postactivation. The described NN topology is well-established for classification tasks \cite{Bishop2006}. 
Since the approximation of the PWC step sizes (cf. Corollary~\ref{cor:alpha_PWC}) by the NN $\Phib(\xb)$ can be interpreted as a classification problem, this topology is the natural choice here.

We trained the NN by minimizing the cost \eqref{eq:NN_cost} with $k=i^\ast+l=12$ using the AdamW \cite{loshchilov2019decoupledweightdecayregularization} optimizer with an initial learning rate of 0.02 and learning rate scheduling. The data set $\Nc$ for training is generated by sampling the dual QP on a regular grid with $50$ sampling steps in each dimension of the parameter space, i.e., for $\xb_1$ and $\xb_2$, where we only kept the samples for which the dual QP is feasible. 
This results in a data set with 2314 samples. After training the NN, we evaluated the learning-enhanced DPGD solver as illustrated in Figure~\ref{fig:grad_steps} for different values of $i^\ast$. The results are summarized in Figure~\ref{fig:step_evaluations},
where we show the average value of $||\mub^\ast-\mub^k||_2$ after $k=i^\ast+l=i^\ast+2$ DPGD steps for 231 randomly chosen $\xb$ that have not been used during training. We compared the proposed learning-enhanced DPGD solver shown in Figure~\ref{fig:NNs}, where the step sizes are given by \eqref{eq:alpha_NN} with a standard DPGD solver with constant step sizes $\alpha_i=\alpha_{\text{std}}$ for all $i\in\{0,\dots,i^{\ast}+l-1\}$. Moreover, in Figure~\ref{fig:opt_alphas} we show the optimal step sizes \eqref{eq:alphaLambda}, required for exactly solving the dual QP \eqref{eq:mpQP_dual}, and the step sizes learned by the NN $\Phib(\xb)$ for some samples $\xb=\xb^j$. We can make several interesting observations based on the figures.

\begin{figure}[t]
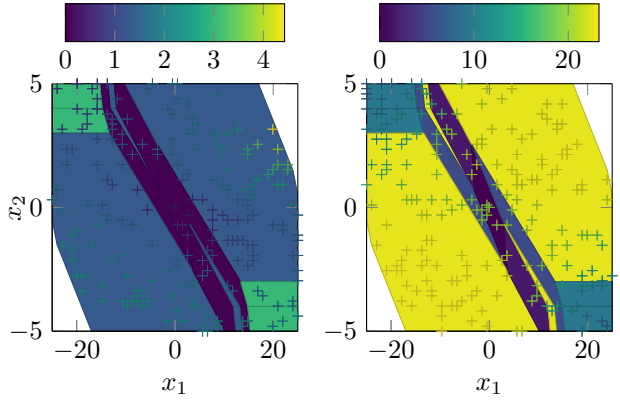

    \centering
    \begin{subfigure}[t]{0.26\textwidth}
        \centering
        \vspace{2mm}
        \label{fig:alpha0}
        \input{alpha_0}
    \end{subfigure}
    \hspace{-8.5mm}
    \begin{subfigure}[t]{0.26\textwidth}
        \centering
        \vspace{2mm}
        \label{fig:alpha1}
        \input{alpha_1}
    \end{subfigure}
    \vspace{-2mm}
    \caption{The optimal PWC step sizes $1/\lambda_{i-i^\ast+1}$ according to \eqref{eq:alphaLambda} are illustrated for $i=i^\ast$ in the left figure and $i=i^\ast+1$ in the right figure. The predictions by the NN $\Phib_{i}(\xb^j)$ for different validation samples $\xb^j$ are represented by the $+$ symbols in (a) and (b) for $i=i^\ast$ and $i=i^\ast+1$, respectively.}
    \label{fig:opt_alphas}
\end{figure}

First, we can clearly see in Figure~\ref{fig:opt_alphas} that, as stated in Corollary~\ref{cor:alpha_PWC}, the optimal step sizes are PWC functions. Moreover, although we do not directly enforce the learning of the PWC optimal step sizes \eqref{eq:alphaLambda}, since we only minimize \eqref{eq:NN_cost} during training, the NN still learns the proposed PWC optimal step sizes. This results in larger step sizes than the classical approach with constant step sizes, which is consistent with Corollary~\ref{cor:optStepsLarger}. Here, $\alpha_{\text{std}}=0.413$, while the optimal step sizes, resulting from the eigenvalues of $\Mb_{\Sc\Sc}$, are between $1.24$ and $22.39$. Moreover, we observe in Figure~\ref{fig:step_evaluations} that even for small $i^\ast\geq2$ and thus for a small number of solver iterations, the proposed learning-enhanced DPGD solver yields better results. 
The higher approximation error for $i^\ast<2$ may result from the fact that, as discussed in Section~\ref{sec:validityAssumptions}, for small $i^\ast$ it is more likely that Assumption~\ref{assum:lastSteps}, which is crucial for the results from Section~\ref{sec:opt_steps} to hold, does not hold. Thus, in this case, we may not be able to compute $\mub^\ast$ exactly by the DPGD solver. However, as apparent from Figure~\ref{fig:step_evaluations}, this undesired effect diminishes already for small $i^\ast$. For the case study considered, the break-even point at which the proposed method outperforms a standard DPGD solver with constant step sizes is already at $i^\ast=2$. We further observe in the experiments that the fraction of samples $\mub^{i^\ast}(\xb^j)$ satisfying Assumption~\ref{assum:lastSteps} increases rapidly with growing $i^\ast$. For $i^\ast=2$ only $25\%$ of the samples satisfy the assumption whereas for $i^\ast=10$ already over $90\%$ of the samples satisfy Assumption~\ref{assum:lastSteps}. Meaning that for $i^\ast=10$ we can compute the exact optimal solution $\mub^\ast(\xb^l)$ for $90\%$ of the samples $\xb^l$ by using the PWC step size \eqref{eq:alphaLambda}.     

\section{Conclusion and summary}
Based on a rigorous theoretical analysis of DPGD solvers for mpQPs in Section~\ref{sec:opt_steps} and Section~\ref{sec:validityAssumptions}, we derived NN topologies that allow us to enhance the standard DPGD solver by approximating optimal parameter-dependent step sizes. The theoretical analysis reveals several new insights into DPGD solvers for mpQPs. We showed that the dual of mpQPs and thereby the mpQPs themselves can be solved exactly by these solvers when choosing parameter-dependent optimal step sizes. Moreover, we showed that these optimal step sizes are PWC functions of the mpQPs parameter.  
In a case study, we further showed that the derived NN topologies are not just theoretically sound but also outperform standard DPGD solvers with constant step sizes. 

\bibliographystyle{unsrt}        

\begin{thebibliography}{10}

\bibitem{Rawlings2017}
J.~B. Rawlings, D.~Q. Mayne, and M.~M. Diehl.
\newblock {\em Model Predictive Control: Theory, Computation, and Design}.
\newblock Nob Hill Publishing, 2nd edition edition, 2017.

\bibitem{Bemporad2002}
A.~Bemporad, M.~Morari, V.~Dua, and E.~N. Pistikopoulos.
\newblock The explicit linear quadratic regulator for constrained systems.
\newblock {\em Automatica}, 38(1):3--20, 2002.

\bibitem{Chen2018}
S.~{Chen}, K.~{Saulnier}, N.~{Atanasov}, D.~D. {Lee}, V.~{Kumar}, G.~J. {Pappas}, and M.~{Morari}.
\newblock Approximating explicit model predictive control using constrained neural networks.
\newblock In {\em Proc. 2018 American Control Conference}, pages 1520--1527, 2018.

\bibitem{Karg2020}
B.~{Karg} and S.~{Lucia}.
\newblock Efficient representation and approximation of model predictive control laws via deep learning.
\newblock {\em IEEE Transactions on Cybernetics}, 50(9):3866--3878, 2020.

\bibitem{Fabiani2022}
F.~Fabiani and P.J. Goulart.
\newblock Reliably-stabilizing piecewise-affine neural network controllers.
\newblock {\em IEEE Transactions on Automatic Control}, pages 1--15, 2022.

\bibitem{Teichrib2024CDC}
D.~Teichrib and M.~Schulze~Darup.
\newblock Reachability analysis for piecewise affine systems with neural network-based controllers.
\newblock In {\em Proc. of the 2024 Conference on Decision and Control}, pages 894--900, 2024.

\bibitem{Bhardwaj2020}
M.~Bhardwaj, S.~Choudhury, and B.~Boots.
\newblock Blending mpc \& value function approximation for efficient reinforcement learning.
\newblock arXiv:2012.05909v1 [cs.LG], 2020.

\bibitem{Zhong2013}
M.~{Zhong}, M.~{Johnson}, Y.~{Tassa}, T.~{Erez}, and E.~{Todorov}.
\newblock Value function approximation and model predictive control.
\newblock In {\em 2013 IEEE Symposium on Adaptive Dynamic Programming and Reinforcement Learning}, pages 100--107, 2013.

\bibitem{Chen2021}
S.~Chen, M.~Fazlyab, M.~Morari, G.~J. Pappas, and V.~M. Preciado.
\newblock Learning lyapunov functions for hybrid systems.
\newblock In {\em Proc. of the International Conference on Hybrid Systems: Computation and Control}. Association for Computing Machinery, 2021.

\bibitem{Ichnowski2021}
J.~Ichnowski, P.~Jain, B.~Stellato, G.~Banjac, M.~Luo, F.~Borrelli, J.E. Gonzalez, I.~Stoica, and K.~Goldberg.
\newblock Accelerating quadratic optimization with reinforcement learning.
\newblock In {\em Proc. of the 2021 International Conference on Neural Information Processing Systems}, NIPS '21, Red Hook, NY, USA, 2021. Curran Associates Inc.

\bibitem{Luken2024}
L.~L\"uken and S.~Lucia.
\newblock Learning iterative solvers for accurate and fast nonlinear model predictive control via unsupervised training.
\newblock In {\em Proc. of the 2024 European Control Conference}, pages 1843--1850, 2024.

\bibitem{Chen2022}
T.~Chen, X.~Chen, W.~Chen, Z.~Wang, H.~Heaton, J.~Liu, and W.~Yin.
\newblock Learning to optimize: a primer and a benchmark.
\newblock {\em J. Mach. Learn. Res.}, 23(1), January 2022.

\bibitem{kotary2024}
J.~Kotary and F.~Fioretto.
\newblock Learning constrained optimization with deep augmented lagrangian methods, 2024.

\bibitem{Luken2025}
L.~Lüken and S.~Lucia.
\newblock Self-supervised learning of iterative solvers for constrained optimization, 2025.

\bibitem{Teichrib2023}
D.~Teichrib and M.~Schulze~Darup.
\newblock Error bounds for maxout neural network approximations of model predictive control.
\newblock {\em IFAC-PapersOnLine}, 56(2):10113--10119, 2023.
\newblock Proc. of the 2023 IFAC World Congress.

\bibitem{Teichrib2025OJCSYS}
D.~Teichrib and M.~Schulze~Darup.
\newblock On the representation of piecewise quadratic functions by neural networks.
\newblock {\em IEEE Open Journal of Control Systems}, 4:447--462, 2025.

\bibitem{Boyd2004}
S.~Boyd and L.~Vandenberghe.
\newblock {\em Convex Optimization}.
\newblock Cambridge University Press, 2004.

\bibitem{tensorflow2015-whitepaper}
M.~Abadi et~al.
\newblock {TensorFlow}: Large-scale machine learning on heterogeneous systems, 2015.
\newblock Software available from tensorflow.org.

\bibitem{PyTorch}
A.~Paszke et~al.
\newblock {\em PyTorch: an imperative style, high-performance deep learning library}.
\newblock Curran Associates Inc., Red Hook, NY, USA, 2019.

\bibitem{scikit-learn}
F.~Pedregosa et~al.
\newblock Scikit-learn: Machine learning in {P}ython.
\newblock {\em Journal of Machine Learning Research}, 12:2825--2830, 2011.

\bibitem{Teichrib2023CDC}
D.~Teichrib and M.~Schulze~Darup.
\newblock Efficient computation of lipschitz constants for mpc with symmetries.
\newblock In {\em Proc- of the 2023 IEEE Conference on Decision and Control}, pages 6685--6691, 2023.

\bibitem{Gutman1987}
P.-O. Gutman and M.~Cwikel.
\newblock An algorithm to find maximal state constraint sets for discrete-time linear dynamical systems with bounded controls and states.
\newblock {\em IEEE Transactions on Automatic Control}, 32(3):251--254, 1987.

\bibitem{Bishop2006}
Christopher~M. Bishop.
\newblock {\em Pattern Recognition and Machine Learning (Information Science and Statistics)}.
\newblock Springer-Verlag, Berlin, Heidelberg, 2006.

\bibitem{loshchilov2019decoupledweightdecayregularization}
I.~Loshchilov and F.~Hutter.
\newblock Decoupled weight decay regularization, 2019.

\end{thebibliography}

\end{document}